\documentclass[12pt]{amsart}
\usepackage{amsmath, amsthm, amsfonts, amssymb}
\newtheorem{theorem}{Theorem}[section]

\newtheorem{lemma}[theorem]{Lemma}
\theoremstyle{definition}

\newtheorem{remark}[theorem]{Remark}

\begin{document}
\title[A result on the sum of element orders of a finite group]{ A result on the sum of element\\  orders of a finite group}
\author[  A. Bahri, B. Khosravi, Z. Akhlaghi ]{ Afsane Bahri, Behrooz Khosravi and Zeinab Akhlaghi }
\address{ Dept. of Pure  Math.,  Faculty  of Math. and Computer Sci. \\
Amirkabir University of Technology (Tehran Polytechnic)\\ 424,
Hafez Ave., Tehran 15914, Iran \newline }
\email{ afsanebahri@aut.ac.ir}
\email{ khosravibbb@yahoo.com}
\email{ z\_akhlaghi@aut.ac.ir}

\thanks{}


\begin{abstract}
Let $G$ be a finite group and $\psi(G)=\sum_{g\in{G}}{o(g)}$. There are some results about the relation between $\psi(G)$ and the structure of $G$. For instance, it is proved that if $G$ is a group of order $n$ and $\psi(G)>\dfrac{211}{1617}\psi(C_n)$, then $G$ is solvable. Herzog {\it{et al.}} in [Herzog {\it{et al.}}, Two new criteria for solvability of finite groups, J. Algebra, 2018] put forward the following conjecture:

\noindent{\bf Conjecture.} {\it {If $G$ is a non-solvable group of order $n$, then
$${\psi(G)}\,{\leq}\,{{\dfrac{211}{1617}}{\psi(C_n)}}$$
with equality if and only if $G=A_5$. In particular, this inequality holds for all non-abelian simple groups.} }

In this paper, we prove a modified version of Herzog's Conjecture.
\end{abstract}

\maketitle
\section{\bf Introduction}

 Let $G$ be a finite group and $\pi(G)$ be the set of all prime divisors of $|G|$. Let $\psi(G)=\sum_{g\in{G}}{o(g)}$. In \cite{isaacs}, it is proved that cyclic groups are characterized by their orders and the sum of element orders.
In fact, they proved that the maximum amount of this function occurs on cyclic groups among all groups of the same order. Many authors have investigated some other properties of this function (see \cite{dr,2011,upperbound,twocriteria,2m,2013,2014,2015}). The following conjecture was posed by Herzog {\it{et al.}} in \cite{twocriteria} :

\noindent\textbf{Conjecture.}
{\it{If $G$ is a non-solvable group of order $n$, then
$${\psi(G)}\,{\leq}\,{{\dfrac{211}{1617}}{\psi(C_n)}}$$
with equality if and only if $G=A_5$. In particular, this inequality holds for all non-abelian simple groups.}}

It is proved that if $G$ is a group of order $n$ and $\psi(G)>\dfrac{211}{1617}\psi(C_n)$, then $G$ is solvable (see \cite{dr},\cite{twocriteria}). Baniasad Azad and Khosravi in \cite{dr} gave some other groups the equality holds for. Actually, they showed that $A_5{\times}C_m$, where $(30,m)=1$, satisfies the equality. Here, we prove that these groups are the only groups the equality holds for. Thus, we prove a modified version of Conjecture 6 in \cite{twocriteria} as follows.

{\bf Main Theorem.} {\it {Suppose that $G$ is a non-solvable group of order $n$ and $\psi(G)=\dfrac{211}{1617}\psi(C_n)$. Then $G={A_5}\times{C_m}$, where $(30,m)=1$. }}
%%%%%%%%%%%%%%%%%%%%%%%%%%%%%%%%%%%%
\section{\bf{Preliminary Results}}

\begin{lemma} \cite[Corollary B]{isaacs}\label{sylownormal}
Let $P\in{{\rm Syl}_p(G)}$, and assume that $P\,{\lhd}\,G$ and that $P$ is cyclic. Then $\psi(G)\,\leq\,{\psi(P)}{\psi(G/P)}$, with equality if and only if $P$ is central in $G$.
\end{lemma}
%%%%%%%%%%%%%%%%%%%%%%%%%%%%%%%%%%%%%%%%%%%%%%%%%%%%%%%%%
\begin{lemma} \cite[Proposition 2.6]{twocriteria}\label{subgroup}
Let $H$ be a normal subgroup of the finite group $G$. Then $\psi(G)\,\leq\,{\psi(G/H)}{|H|^2}$.
\end{lemma}
%%%%%%%%%%%%%%%%%%%%%%%%%%%%%%%%%%%%%%%%%%%%%%%%%%%%%%%%%%%%%%
\begin{lemma} \cite[Lemma 2.1]{2011}\label{directproduct}
If $G$ and $H$ are finite groups, then ${\psi(G{\times}H)}\,\leq\,{\psi(G)}{\psi(H)}$. Also $\psi(G{\times}H)={\psi(G)}{\psi(H)}$ if and only if ${\rm gcd}(|G|,|H|)=1$.
\end{lemma}
%%%%%%%%%%%%%%%%%%%%%%%%%%%%%%%%%%%%%%%%%%%%%%%%%%%%%%%%
\begin{lemma} \cite[Proposition 2.5]{upperbound}\label{2p}
Let $G$ be a finite group and suppose that there exists $x\in{G}$ such that $|G:\langle{x}\rangle|\,<\,2p$, where $p$ is the maximal prime divisor of $|G|$. Then one of the following holds:
\begin{itemize}
\item[(i)] $G$ has a normal cyclic Sylow $p$-subgroup,
\item[(ii)] $G$ is solvable and $\langle{x}\rangle$ is a maximal subgroup of $G$ of index either $p$ or $p+1$.
\end{itemize}
\end{lemma}
%%%%%%%%%%%%%%%%%%%%%%%%%%%%%%%%%%%%%%%%%%%%
\begin{lemma} \cite[Lemma 2.9]{upperbound}\label{cyclicsum}

\begin{itemize}
\item[(1)] {If $P$ is a cyclic group of order $p^r$ for some prime $p$, then
$$\psi(P)=\dfrac{{p^{2r+1}}+1}{p+1}=\dfrac{p|P|^2+1}{p+1}.$$}
\item[(2)] {Let ${p_1}<{p_2}<{\cdots}<{{p_t}=p}$ be the prime divisors of $n$ and denote the corresponding Sylow subgroups of $C_n$ by ${P_1},{P_2},{\cdots},P_t$. Then
$$\psi(C_n)=\prod_{i=1^t}{\psi(P_i)}\geq{\dfrac{2}{p+1}}n^2.$$}
\end{itemize}
\end{lemma}
%%%%%%%%%%%%%%%%%%%%%%%%%%%%%%%%%%%%%%%%%%%%%%%%%%%%%%%%%%%%
\begin{lemma} \cite[Theorem 1]{twocriteria}\label{A/B}
Let $G$ be a finite group of order n containing a subgroup $A$ of prime power index $p^s$. Suppose that $A$ contains a normal cyclic subgroup $B$ satisfying the following condition: $A/B$ is a cyclic group of order $2^r$ for some non-negative integer $r$. Then $G$ is a solvable group.
\end{lemma}

\begin{lemma} \cite[Theorem]{Herstein}\label{maximalabeli}
Let $G$ be a finite group, $A$ an Abelian subgroup of $G$. If $A$ is a maximal subgroup of $G$ then $G$ is solvable.
\end{lemma}
%%%%%%%%%%%%%%%%%%%%%%%%%%%%%%%%%%%%%%%%%%%%%%%%%%%%%%%%%%%%%%%%%%%%
\begin{lemma} \cite[Lemma 9.1]{finiteisaacs}\label{G'}
Let $G$ be a group, and  suppose that $G/Z(G)$ is simple. Then $G/Z(G)$ is nonabelian, and $G'$ is perfect. Also $G'/Z(G')$ is isomorphic to the simple group $G/Z(G)$.
\end{lemma}
%%%%%%%%%%%%%%%%%%%%%%%%%%%%%%%%%%%%%%%%%%%%%%%%%%%%%%%%%
\begin{lemma} \cite[Corollary 5.14]{finiteisaacs}\label{minsylow}
Let $P\in{{\rm Syl}_p(G)}$, where $G$ is a finite group and $p$ is the smallest prime divisor of $|G|$, and assume that $P$ is cyclic. Then $G$ has a normal $p$-complement.
\end{lemma}
%%%%%%%%%%%%%%%%%%%%%%%%%%%%%%%%%%%%%%%%%%%%%%%%%%%%%%%%%%%%%%
\begin{lemma} \cite[Theorem 2.20]{finiteisaacs}\label{core}
Let $A$ be a cyclic proper subgroup of a finite group $G$, and let $K={\rm core}_G(A)$. Then $|A:K|\,<\,|G:A|$, and in particular, if $|A|\,\geq\,|G:A|$, then $K\,>\,1$
\end{lemma}
%%%%%%%%%%%%%%%%%%%%%%%%%%%%%%%%%%%%%%%%%%%%%%%%%%%%%%%%%%%%%%%
\begin{lemma} \cite[Theorem 3.1]{hall}\label{r}
If $n=1+rp$, with $1\,<r<\,(p+3)/2$ there is not a group $G$ with $n$ Sylow $p$-subgroups unless $n=q^t$ where $q$ is a prime, or $r=(p-3)/2$ and $p\,>\,3$ is a Fermat prime.
\end{lemma}
%%%%%%%%%%%%%%%%%%%%%%%%%%%%%%%%%%%%%%%%%%%%%%%%%%%%%%%%%%%%%%%%
\begin{lemma} \cite[Theorem 3.2]{hall}\label{sylownumber}
There is no group $G$ with $n_3=22$, with $n_5=21$, or with $n_p=1+3p$ for $p\,\geq\,7$.
\end{lemma}
%%%%%%%%%%%%%%%%%%%%%%%%%%%%%%%%%%%%%%%%%%%%%%%%%%%%%%%%%%%%%%
\begin{lemma} \cite[Lemma 2.1]{dr}\label{2.1dr}
Let $G$ be a group of order $n={p_1}^{\alpha_1}\cdots{p_k}^{\alpha_k}$, where $p_1,\cdots,p_k$ are distinct primes. Let $\psi(G)\,>\,{\dfrac{r}{s}{\psi(C_n)}}$, for some integers $r,s$. Then there exists a cyclic subgroup $\langle{x}\rangle$ such that
$$[G:\langle{x}\rangle]\,<\,{\dfrac{s}{r}}\cdot{\dfrac{p_1+1}{p_1}}\cdots{\dfrac{p_k+1}{p_k}}.$$
\end{lemma}
%%%%%%%%%%%%%%%%%%%%%%%%%%%%%%%%%%%%%%%%%%%%%%%%%%%%%%%%%
%\begin{lemma} \cite[Lemma 2.2]{dr}\label{2.2dr}
%Let $p$ be a prime number and $a,b>0$. Then ${\psi(C_{p^{a+b}})}\,\geq\,{\psi(C_{p^a})}{\psi(C_{p^{b}})}$.
%\end{lemma}
%%%%%%%%%%%%%%%%%%%%%%%%%%%%%%%%%%%%%%%%%%%%%%%%%%%%%%%%%%%%%
\begin{lemma} \cite[Lemma 2.3]{dr}\label{2.3dr}
\begin{itemize}
\item[(a)] Let $p\in\{2,3,5\}$ and $a>0$. Then $p^{2a}>{{\dfrac{13}{12}}{\psi(C_{p^a})}}$.
\item[(b)] Let $\pi(m)\subseteq\{2,3,5\}$ and $m\geq2$. Then ${m^2}>{\dfrac{13}{12}{\psi(C_m)}}$.
\end{itemize}
\end{lemma}
%%%%%%%%%%%%%%%%%%%%%%%%%%%%%%%%%%%%%%%%%%%%%%%%%%%%%%%%%%%%%
\begin{lemma} \cite[Lemma 2.4]{dr}\label{2.4dr}
Let $n={p_1}^{\alpha_1}{p_2}^{\alpha_2}\cdots{p_r}^{\alpha_r}$ be a positive integer, where $p_i$ are primes, ${p_1}<{p_2}<\cdots<{p_r}=p$ and ${\alpha_i}>0$, for each $1\,{\leq}\,i\,{\leq}\,r$. If $p\geq{13}$, then
$${\psi(C_n)}\geq{{\dfrac{5005}{1152}}{\dfrac{n^2}{p+1}}}.$$
\end{lemma}

By the proof given in \cite{dr} for Lemma 2.2, we have:
%%%%%%%%%%%%%%%%%%%%%%%%%%%%%%%%%%%%%%%%%%%%%%%%%%%%%%%%%%%%%%%%
\begin{lemma} \label{big}
Let $p$ be a prime number and $a,b>0$. Then ${\psi(C_{p^{a+b}})}>{\psi(C_{p^a})}{\psi(C_{p^{b}})}$.
\end{lemma}
%%%%%%%%%%%%%%%%%%%%%%%%%%%%%%%%%%%%%%%%%%%%%%%%%%%%%
\section{\bf{Main Results}}
In this section we prove the modified version of Herzog's Conjecture.
\begin{lemma}\label{nonsolvablemyselfbozorg}
Let $G$ be a non-solvable group of order n and $p\mid{n}$. Furthermore, $\psi(G)={\dfrac{211}{1617}{\psi(C_n)}}$ and $P\in{{\rm Syl}_p(G)}$ is cyclic and normal in $G$. Then $P$ has a normal $p$-complement $K$ in $G$, where $\psi(K)=\dfrac{211}{1617}\psi(C_{|K|})$.
\end{lemma}
\begin{proof}
By Lemma \ref{sylownormal}, we have
\begin{equation}
\label{2}
\psi(P)\psi(G/P)\geq{\psi(G)}=\dfrac{211}{1617}\psi(C_n)=\dfrac{211}{1617}{\psi(C_{|P|})\psi(C_{|G/P|})}.
\end{equation}
Therefore
$$\psi(G/P)\geq\dfrac{211}{1617}{\psi(C_{|G/P|})}.$$
Now by the non-solvability of $G$ and the main result of \cite{dr}, we conclude that $\psi(G/P)=\dfrac{211}{1617}{\psi(C_{|G/P|})}$. It follows that the equality holds in (\ref{2}). Thus $P\,\leq\,Z(G)$ by Lemma \ref{sylownormal}. Therefore by Burnside's normal $p$-complement theorem, there exists $K\,\unlhd\,G$ such that $G=P\times{K}$.
\end{proof}
%%%%%%%%%%%%%%%%%%%%%%%%%%%%%%%%%%%%%%%%%%%%%%%%%

%Suppose that $G$ is a non-solvable group of order $n$ with ${\psi(G)}={{\dfrac{211}{1617}}\psi(C_n)}$. By the non-solvability of $G$ and Burnside's $p^aq^b-$theorem, $|\pi(G)|\geq3$. Hence $p\geq5$, where $p$ is the largest prime divisor of $n$. We attempt to prove the theorem in the following four cases:
%%%%%%%%%%%%%%%%%%%%%%%%%%%%%%%%%%%%%%%%%%%%%%%%
\theorem{ \label{p5}
Suppose that $G$ is a non-solvable group of order $n$, and $p=5$ is the largest prime divisor of $n$. Furthermore, $\psi(G)=\dfrac{211}{1617}\psi(C_n)$. Then $G=A_5$.}
\begin{proof}
By the assumption $n=2^{\alpha}3^{\beta}5^{\gamma}$. We note that $G$ has no normal Sylow subgroup. By Lemma \ref{cyclicsum}, we have
\begin{equation}
\label{6}
\psi(G)={\dfrac{211}{1617}}{\psi(C_n)}>{\dfrac{211}{1617}}\cdot{\dfrac{2^{2{\alpha}+1}}{3}}\cdot{\dfrac{3^{2{\beta}+1}}{4}}\cdot{\dfrac{5^{2{\gamma}+1}}{6}}={\dfrac{211}{1617}}\cdot{\dfrac{5}{12}}n^2.
\end{equation}
Therefore, there exists $x\in{G}$ such that $[G:\langle{x}\rangle]<19$.

If $[G:\langle{x}\rangle]<10$, then by the non-solvability of $G$ and Lemma \ref{2p}, $G$ has a normal Sylow $5$-subgroup, which is a contradiction. Hence $10\leq[G:\langle{x}\rangle]<19$. Now we consider the following cases:
\begin{itemize}

%%%%%%%%%%%%%%%%%%%%%%%%%%%%%%%%%%%%%%%%%%%%%%%%%%%%%%%%
\item[{\bf(a)}] Let {$[G:\langle{x}\rangle]=18$.

Then there exists $P_5\in{{\rm Syl}_5(G)}$ such that $P_5\,\leq\,\langle{x}\rangle$. Therefore
$$18=[G:\langle{x}\rangle]=[G:N_G(P_5)][N_G(P_5):\langle{x}\rangle]=(1+5k)[N_G(P_5):\langle{x}\rangle].$$
Hence $[G:N_G(P_5)]=6$ and $[N_G(P_5):\langle{x}\rangle]=3$. Let $H={\rm core}_G(\langle{x}\rangle)$. Now by Lemma \ref{core}, $[\langle{x}\rangle{:H}]<[G:\langle{x}\rangle]=18$. If $5\nmid{[\langle{x}\rangle:H]}$, then $H$ contains a Sylow $5$-subgroup of $G$, say $P_5$, implying $P_5{\unlhd}G$, a contradiction. Hence $5\mid{[\langle{x}\rangle{:H}]}$. By the non-solvability of $G/H$, $[\langle{x}\rangle{:H}]=10$ and $|G/H|=180$. Since $A_5{\times}C_3$ is the only non-solvable group of order $180$, we conclude that $G/H{\cong}A_5{\times}C_3$ and $\psi(G/H)=1237$. Now using Lemma \ref{subgroup} and (\ref{6}), we have
$${\dfrac{211}{1617}\cdot{\dfrac{5}{12}}n^2}<{\psi(G)}\leq{\psi(G/H)}{|H|^2}=1237(n/180)^2,$$
which is a contradiction.}
%%%%%%%%%%%%%%%%%%%%%%%%%%%%%%%%%%%%%%%%%%%%%%%%%%%%%%%%
\item[{\bf(b)}] Let {$[G:\langle{x}\rangle]=16$.

By Lemma \ref{A/B}, we conclude that G is solvable contradicting our hypothesis.
}
%%%%%%%%%%%%%%%%%%%%%%%%%%%%%%%%%%%%%%%%%%%%%%%
\item[{\bf(c)}] Let {$[G:\langle{x}\rangle]=15$.

Similar to the previous cases, $\langle{x}\rangle$ contains a Sylow $2$-subgroup of $G$, say $P_2$. Using Lemma \ref{minsylow}, $G$ has a normal $2$-complement, say $N$. Hence by the Feit-Thompson theorem, $N$ is solvable contradicting the non-solvability of $G$.
}

%%%%%%%%%%%%%%%%%%%%%%%%%%%%%%%%%%%%%%%%%%%%%%%%%%
\item[{\bf(d)}] Let {$[G:\langle{x}\rangle]=10$.

\noindent There exists $P_3\in{{\rm Syl}_3(G)}$ such that $P_3\leq{\langle{x}\rangle}$. Then $\langle{x}\rangle\,{\leq}\;{N_G(P_3)}$ and we have
$$10=[G:\langle{x}\rangle]=[G:N_G(P_3)][N_G(P_3):\langle{x}\rangle]=(1+3k)[N_G(P_3):\langle{x}\rangle].$$
Then $N_G(P_3)=\langle{x}\rangle$. We claim that $N_G(P_3)$ is a maximal subgroup of $G$. Suppose that $N_G(P_3)<L\leq{G}$. Hence $N_L(P_3)=N_G(P_3)$ and $10=[G:L][L:N_L(P_3)]$ implying $[G:L]=1$. Now Lemma \ref{maximalabeli} implies that $G$ is solvable, a contradiction.  }
%%%%%%%%%%%%%%%%%%%%%%%%%%%%%%%%%%%%%%%%%%%%%%
\item[{\bf(e)}] Let {$[G:\langle{x}\rangle]=12$.

There exists $P_5\in{{\rm Syl}_5(G)}$ such that $P_5\leq{\langle{x}\rangle}$. As we argued, we have
$$12=(1+5k)[N_G(P_5):\langle{x}\rangle].$$
Thus $[G:N_G(P_5)]=6$ and $[N_G(P_5):\langle{x}\rangle]=2$. Let $H={\rm core}_G(\langle{x}\rangle)$. Hence by Lemma \ref{core}, $[\langle{x}\rangle:H]<[G:\langle{x}\rangle]=12$.

Similarly to Case (a), $5\mid[\langle{x}\rangle:H]$, and so $[\langle{x}\rangle:H]=5$ or $10$. Now we consider each case:
%%%%%%%%%%%%%%%%%%%%%%%%%%%%%%%%%%%%%%

$\bullet$ Let $[\langle{x}\rangle:H]=10$.

In this case, $|G/H|=120$ and by the non-solvability of $G/H$, we have $\psi(G/H)\leq{663}$. By Lemma \ref{subgroup} and (\ref{6}), we conclude that:
$$\dfrac{211}{1617}\cdot\dfrac{5}{12}n^2<\psi(G)\leq{\psi(G/H)}{|H|^2}\leq663(n/120)^2,$$
which is a contradiction.

%%%%%%%%%%%%%%%%%%%%%%
$\bullet$ Let {$[\langle{x}\rangle:H]=5$}.

In this case, $|G/H|=60$, and so $G/H{\cong}A_5$. As a result, $H$ is a maximal normal subgroup of $G$. On the other hand, $H$ is cyclic, so $H\,\leq\,C_G(H)$. Then $H=C_G(H)$ or $G=C_G(H)$. The Normalizer-Centralizer Theorem implies that $G=C_G(H)$, and so $H\,\leq\,{Z(G)}$. Therefore $H=Z(G)$ and $G/Z(G){\cong}A_5$. Hence ${G'Z(G)/Z(G)}\cong{A_5}$ implying $G=G'Z(G)$, and we have $\dfrac{G'}{G'{\cap}Z(G)}{\cong}A_5$. On the other hand, Lemma \ref{G'} implies that $G'$ is perfect and $G'
/Z(G')\cong{A_5}$. By the fact that the Schur multiplier of $A_5$ is equal to 2, we have $G'\cong{{\rm SL}(2,5)}$ or $G'\cong{A_5}$.

If $G'{\cong}{\rm SL}(2,5)$, then $|G'{\cap}Z(G)|=2$ and we have
$$\dfrac{G}{G'{\cap}Z(G)}\cong{\dfrac{G'Z(G)}{G'{\cap}Z(G)}}\cong{\dfrac{G'}{G'{\cap}Z(G)}}{\times}{\dfrac{Z(G)}{G'{\cap}Z(G)}}\cong{A_5}\times{C_m},$$
where $m=n/120$. So by Lemmas \ref{subgroup} and \ref{directproduct}, we conclude that
$$\psi(G)\leq{\psi(\dfrac{G}{G'{\cap}Z(G)})}{|G'{\cap}Z(G)|^2}\leq{\psi(A_5)\psi(C_m)}{|G'{\cap}Z(G)|^2}=4\cdot211{\psi(C_m)}.$$
Now by (\ref{6}) and some simple calculations, we obtain $72000m^2<77616\psi(C_m)<78000\psi(C_m)$. Hence $12m^2<13\psi(C_m)$. By Lemma \ref{2.3dr}, $m=1$, $n=120$ and $G={\rm SL}(2,5)$. It means that, $663=\psi(G)=\dfrac{211}{1617}\psi(C_{120})>824$, a contradiction.

If $G'\cong{A_5}$, then $|G'{\cap}Z(G)|=1$ and $G={G'}\times{Z(G)}\cong{A_5}\times{Z(G)}$.

If $1\neq Z(G)$, then $|Z(G)|={2^{\alpha-2}}{3^{\beta-1}}{5^{\gamma-1}}$ and by Lemmas \ref{directproduct} and \ref{big}, we have
\begin{align*}
\psi(G)=\psi(A_5{\times}Z(G))&\leq{\psi{(A_5)}\psi{(Z(G))}} = 211{\psi(C_{2^{\alpha-2}})\psi(C_{3^{\beta-1}})\psi(C_{5^{\gamma-1}})} \\
&< {211}\cdot{\dfrac{\psi(C_{2^{\alpha}})}{\psi(C_4)}}\cdot{\dfrac{\psi(C_{3^{\beta}})}{\psi(C_3)}}\cdot{\dfrac{\psi(C_{5^\gamma})}{\psi(C_5)}}= \dfrac{211}{1617}{\psi(C_n)},
\end{align*}
which is a contradiction. That is why, we must have $Z(G)=1$ and $G={A_5}$.

}
%%%%%%%%%%%%%%%%%%%%%%%%%%%%%%%%%%%%%%
\end{itemize}
\end{proof}
%%%%%%%%%%%%%%%%%%%%%%%%%%%%%%%%%%%%%%%%%%%%%%%%%%
\remark{Let $G$ be a non-solvable group. Then there exists a non-abelian simple group $S$ such that $|S|\mid |G|$. If $3\notin\pi({G})$, then $S$ is a Suzuki group, but there exists no Suzuki group which $\pi(S){\subseteq}\{2,3,5,7,11\}$. Therefore, the order of each non-solvable group $G$, where $\pi(G)\subseteq\{2,3,5,7,11\}$ is divisible by $3$. }
%%%%%%%%%%%%%%%%%%%%%%%%%%%%%%%%%%%%%%%%%%%%%%%%%%%%%%%%%%%
\theorem{\label{p7}
Suppose that $G$ is a non-solvable group of order $n$, and $p=7$ is the largest prime divisor of $n$. Furthermore, $\psi(G)=\dfrac{211}{1617}\psi(C_n)$. Then $G=A_5\times{C_m}$, where $m=7^{\alpha}$, for some $\alpha\geq{1}$.}
\begin{proof}

In this case, we have either $\pi(G)=\{2,3,7\}$ or $\pi(G)=\{2,3,5,7\}$.

\textbf{Case (1)\,}Suppose that $\pi(G)=\{2,3,7\}$.

First of all, we notice that $G$ has no normal Sylow subgroup. Now the same as before, we have
\begin{equation}
\label{7}
\psi(G)=\dfrac{211}{1617}\psi(C_n)>\dfrac{211}{1617}\cdot{\dfrac{7}{16}}n^2.
\end{equation}
Hence there exists $x\in{G}$ such that $|G:\langle{x}\rangle|<18$. By Lemma \ref{2p}, if $[G:\langle{x}\rangle]<14$, then $G$ has a normal Sylow $7$-subgroup, a contradiction. So we must have $14\leq[G:\langle{x}\rangle]<18$.

{\bf(a)\,} Let $[G:\langle{x}\rangle]=14$.

There exists $P_3\in{{\rm Syl}_3(G)}$ such that $P_3\,\leq\,\langle{x}\rangle$. Similarly to the above discussion, we have
$$14=(1+3k)[N_G(P_3):\langle{x}\rangle].$$ Hence $[G:N_G(P_3)]=7$ and $[N_G(P_3):\langle{x}\rangle]=2$. Let $H={\rm core}_G(\langle{x}\rangle)$. Lemma \ref{core} implies that $[\langle{x}\rangle:H]<[G:\langle{x}\rangle]=14$. Then $3\mid[\langle{x}\rangle:H]$. Therefore by the non-solvability of $G/H$, $[\langle{x}\rangle:H]=12$ and $|G/H|=168$. It follows that $G/H{\cong}L_2(7)$ and $\psi(G/H)=715$. Now by Lemma \ref{subgroup} and (\ref{7}), we have
$$\dfrac{211}{1617}\cdot{\dfrac{7}{16}}n^2<\psi(G)\leq{\psi(G/H)}|H|^2=715(n/168)^2,$$
which is a contradiction.

{\bf(b)\,} Let $[G:\langle{x}\rangle]=16$.

Lemma \ref{A/B} implies that $G$ is solvable, a contradiction.

\indent \textbf{Case (2)\,}Suppose that $\pi(G)=\{2,3,5,7\}$.

By the assumption, we have
$$\psi(G)={\dfrac{211}{1617}}\psi(C_n)>{\dfrac{211}{1618}}\psi(C_n).$$
Now using Lemma \ref{2.1dr}, we conclude that there exists $x\in{G}$ such that
$$[G:\langle{x}\rangle]<{\dfrac{1618}{211}}\cdot{\dfrac{2+1}{2}}\cdot{\dfrac{3+1}{3}}\cdot{\dfrac{5+1}{5}}\cdot{\dfrac{7+1}{7}}<22.$$
If $[G:\langle{x}\rangle]<14$, then by Lemma \ref{2p}, $G$ has a normal cyclic Sylow $7-$subgroup, say $P_7$. Now by Lemma \ref{nonsolvablemyselfbozorg}, there exists $K\unlhd{G}$ such that $G=P_7{\times}K$ and $\psi(K)=\dfrac{211}{1617}\psi(C_{|K|})$. We notice that $K$ is non-solvable and $\pi(K)=\{2,3,5\}$. Hence it satisfies all the hypotheses of Theorem \ref{p5}. Therefore $K=A_5$ and $G=A_5\times{C_m}$, where $m=7^{\alpha}$ with $\alpha\geq1$.

Thus assume that $14\leq[G:\langle{x}\rangle]<22$.

{\bf(a)\,} Let $[G:\langle{x}\rangle]=14$.

Similarly to the above discussion, we have $14=(1+5k)[N_G(P_5):\langle{x}\rangle]$, for some $P_5\in{{\rm Syl}_5(G)}$. Hence $k=0$, and so by Lemma \ref{nonsolvablemyselfbozorg}, there exists a normal $5$-complement $K$, where $K$ is non-solvable, $\pi(K)=\{2,3,7\}$ and $\psi(K)={\dfrac{211}{1617}}\psi(C_{|K|})$. Therefore, $K$ satisfies all the hypotheses of Case (1). Then the same as before, it causes a contradiction.
%%%%%%%%%%%%%%%%%%%%%%%%%%%%%%%%%%%%%%%%%%%%%%%

{\bf(b)\,} Let $[G:\langle{x}\rangle]=16$.

By Lemma \ref{A/B}, $G$ is solvable, a contradiction.

%%%%%%%%%%%%%%%%%%%%%%%%%%%%%%%%%%%%%%%%%%%%%%%%%%

{\bf(c)\,} Let $[G:\langle{x}\rangle]=21$.

There exists $P_5\in{{\rm Syl}_5(G)}$ such that $P_5\leq{\langle{x}\rangle}$. Then $21=(1+5k)[N_G(P_5):\langle{x}\rangle]$, and so $k=0$, by Lemma \ref{sylownumber}. Therefore, Lemma \ref{nonsolvablemyselfbozorg} implies that $G$ contains a normal $5$-complement satisfying all the hypotheses of Case (1). Hence it causes a contradiction.

%%%%%%%%%%%%%%%%%%%%%%%%%%%%%%%%%%%%%%%%%%%%%%%%%%%%%%%%
{\bf(d)\,} Let $[G:\langle{x}\rangle]=15$.

There exists $P_7\in{{\rm Syl}_7(G)}$ such that $P_7\leq{\langle{x}\rangle}$. Then $15=(1+7k)[N_G(P_7):\langle{x}\rangle]$, and so $k=0$, by Lemma \ref{r}. Therefore, Lemma \ref{nonsolvablemyselfbozorg} implies that $P_7$ has a normal $7$-complement $K$, where $G={P_7}\times{K}$ and $K$ satisfies all the hypotheses of Theorem \ref{p5}. Hence $G=A_5{\times}P_{7}=A_5{\times}C_m$, where $m=7^{\alpha}$ and $\alpha{\geq}1$.

%%%%%%%%%%%%%%%%%%%%%%%%%%%%%%%%%%%%%%%%%

{\bf(e)\,} Let $[G:\langle{x}\rangle]=18$ or $20$.

Then exactly similar to part (d), we get that $G={A_5}{\times}C_m$, where $m=7^{\alpha}$ and $\alpha\geq{1}$.

%{\bf(d)\,} Let $[G:\langle{x}\rangle]=18$ or $20$.

%$\langle{x}\rangle$ contains a normal cyclic Sylow $7$-subgroup of $G$, say $P_7$. Then we have:
 %$$[G:\langle{x}\rangle]=(1+7k)[N_G(P_7):\langle{x}\rangle]$$  So $k=0$. Then by Lemma \ref{nonsolvablemyselfbozorg}, there exists $K\unlhd{G}$ satisfying all the hypotheses of Theorem \ref{p5}. As a result, $K=A_5$. Thus $G={K}\times{P_7}=A_5{\times}C_m$, where $m=7^{\alpha}$ with $\alpha{\geq}1$.

\end{proof}
%%%%%%%%%%%%%%%%%%%%%%%%%%%%%%%%%%%%%%%%%%%%%
\theorem{\label{p11}
Suppose that $G$ is a non-solvable group of order $n$, and $p=11$ is the largest prime divisor of $n$. Furthermore, $\psi(G)=\dfrac{211}{1617}\psi(C_n)$. Then $G=A_5\times{C_m}$, where $m=11^{\beta}$, for some $\beta\geq{1}$.}
\begin{proof}
We consider the following cases:

\textbf{Case (1)}\,Let $\pi(G)=\{2,3,11\}$.

We note that in this case $G$ has no normal Sylow subgroup. Moreover, by the assumption, we have
$$\psi(G)={\dfrac{211}{1617}}\psi(C_n)>{\dfrac{211}{1618}}\psi(C_n).$$
Now by Lemma \ref{2.1dr}, there exists $x\in{G}$ such that
$$[G:\langle{x}\rangle]<{\dfrac{1618}{211}}\cdot{\dfrac{2+1}{2}}\cdot{\dfrac{3+1}{3}}\cdot{\dfrac{11+1}{11}}<17<2p.$$
Then Lemma \ref{2p} implies that $G$ has a normal Sylow $11$-subgroup, a contradiction.

\indent \textbf{Case (2)}\,Let $\pi(G)=\{2,3,5,11\}$.

The same as before, there exists $x\in{G}$ such that
$$[G:\langle{x}\rangle]<{\dfrac{1618}{211}}\cdot{\dfrac{3}{2}}\cdot{\dfrac{4}{3}}\cdot{\dfrac{6}{5}}\cdot{\dfrac{12}{11}}<21<2p.$$
By the non-solvability of $G$ and Lemma \ref{2p}, we conclude that there exists $P_{11}\in{{\rm Syl}_{11}(G)}$ which is cyclic and normal in $G$. Now using Lemma \ref{nonsolvablemyselfbozorg}, there exists $K\unlhd{G}$ such that $G=P_{11}\times{K}$ and $\psi(K)=\dfrac{211}{1617}\psi(C_{|K|})$. We note that $K$ is non-solvable and $\pi(K)=\{2,3,5\}$. Thus, it satisfies all the hypotheses of Theorem \ref{p5}. Hence $K=A_5$ and $G=A_5{\times}C_m$, where $m=11^{\beta}$ and $\beta\geq{1}$.

\indent \textbf{Case (3)}\,Let $\pi(G)=\{2,3,7,11\}$.

By Lemma \ref{2.1dr}, there exists $x\in{G}$ such that
$$[G:\langle{x}\rangle]<{\dfrac{1618}{211}}\cdot{\dfrac{3}{2}}\cdot{\dfrac{4}{3}}\cdot{\dfrac{8}{7}}\cdot{\dfrac{12}{11}}<20<2p.$$
Therefore by Lemma \ref{2p}, $G$ has a normal cyclic Sylow $11$-subgroup, say $P_{11}$. Hence using Lemma \ref{nonsolvablemyselfbozorg}, $G$ has a normal subgroup $K$ such that $\psi(K)=\dfrac{211}{1617}\psi(C_{|K|})$ and $K\cong{G/P_{11}}$. As a result, $K$ is non-solvable and $\pi(K)=\{2,3,7\}$. Thus, it satisfies all the hypotheses of Case (1) of Theorem \ref{p7}. So the same as before, it causes a contradiction.

\textbf{Case (4)}\,Let $\pi(G)=\{2,3,5,7,11\}$.

As the above discussion, there exists $x\in{G}$ such that
$$[G:\langle{x}\rangle]<{\dfrac{1618}{211}}\cdot{\dfrac{3}{2}}\cdot{\dfrac{4}{3}}\cdot{\dfrac{6}{5}}\cdot{\dfrac{8}{7}}\cdot{\dfrac{12}{11}}<23.$$
If $[G:\langle{x}\rangle]<22$, then there exists $P_{11}\in{{\rm Syl}_{11}(G)}$ which is cyclic and normal in $G$. Hence by Lemma \ref{nonsolvablemyselfbozorg}, there exists $K\unlhd{G}$ such that $K\cong{G/P_{11}}$ and $\psi(K)=\dfrac{211}{1617}\psi(C_{|K|})$. Thus, $K$ satisfies all the hypotheses of {{Case (2)}} of Theorem \ref{p7}. Therefore $K=A_5{\times}C_s$, where $s=7^{\alpha}$ and $\alpha\geq{1}$. Finally, we conclude that $G=A_5\times{C_m}$, where $m=7^{\alpha}11^\beta$ and $\alpha,\beta\geq{1}$.

If $[G:\langle{x}\rangle]=22$, then $\langle{x}\rangle$ contains a Sylow $7$-subgroup of $G$, say $P_7$. The same as before, $22=(1+7k)[N_G(P_7):\langle{x}\rangle]$. Hence $k=0$, by Lemma \ref{r}, and so there exists $K\unlhd{G}$ satisfying all the hypotheses of Case (2). Thus $G=A_5{\times}{C_m}$, where $m=7^{\alpha}11^{\beta}$ and $\alpha,\beta\geq{1}$.
\end{proof}
%%%%%%%%%%%%%%%%%%%%%%%%%%%%%%%%%%%%%%%%%%%%%%%%%%%%%%%%%%%%%

{\bf Main Theorem.} {\it {Suppose that $G$ is a non-solvable group of order $n$ and $\psi(G)=\dfrac{211}{1617}\psi(C_n)$. Then $G={A_5}\times{C_m}$, where $(30,m)=1$. }}
\begin{proof}
We prove the theorem by induction on the size of $p$, where $p$ is the largest prime divisor of $n$. By the assumption, $|\pi(G)|\geq{3}$, and so $p\geq{5}$. If ${5}\leq{p}\leq11$, then by the above theorems, we have $G=A_5{\times}C_m$, where $(30,m)=1$. Thus, we may assume that $p\geq13$.

%$\psi(C_n)\geq{\dfrac{5005}{1152}\cdot{\dfrac{n^2}{p+1}}}$
Then by Lemma \ref{2.4dr}, we have
$${\psi(G)}={\dfrac{211}{1617}}{\psi(C_n)}\geq{{\dfrac{211}{1617}}\cdot{\dfrac{5005}{1152}\cdot{\dfrac{n^2}{p+1}}}}.$$
It means that, there exists $x\in{G}$ such that
\begin{equation}
\label{5}
[G:\langle{x}\rangle]<{\dfrac{1617}{211}}\cdot{\dfrac{1152}{5005}}(p+1)<2(p+1).
\end{equation}

Assume that $p\mid{[G:\langle{x}\rangle]}$. Then $pm=[G:\langle{x}\rangle]<2(p+1)$, and it follows that $m=1$ or $m=2$. If $m=1$, then using Lemma \ref{maximalabeli}, $G$ is solvable which is a contradiction. If $m=2$, then by (\ref{5}), we have $2p<{\dfrac{1617}{211}}\cdot{\dfrac{1152}{5005}}(p+1)$, and so $p<8$, a contradiction.

Thus, we have $p\nmid{[G:\langle{x}\rangle]}$. Therefore, there exists $P\in{{\rm Syl}_p(G)}$ such that $P\leq{\langle{x}\rangle}$. Similarly to the above discussion, we have
$$2(p+1)>[G:\langle{x}\rangle]=[G:N_G(P)][N_G(P):\langle{x}\rangle]=(1+kp)[N_G(P):\langle{x}\rangle].$$
It follows that $k=0$, $k=1$ or $k=2$. If $k=1$ or $k=2$, then $N_G(P)=\langle{x}\rangle$, and similar to part (d) of Theorem \ref{p5}, $N_G(P)=\langle{x}\rangle$ is a maximal subgroup of $G$. Now using Lemma \ref{maximalabeli}, $G$ is solvable, a contradiction. Thus $k=0$ and $P\unlhd{G}$. By Lemma \ref{nonsolvablemyselfbozorg}, there exists $K\unlhd{G}$ such that $G=P\times{K}$, $K\cong{G/P}$ and $\psi(K)={\dfrac{211}{1617}}{\psi(C_{|K|})}$. Moreover, it is obvious that $K$ is non-solvable. Now by the inductive hypotheses, $K=A_5{\times}C_s$, where $(30,s)=1$. Thus $G=A_5{\times}C_m$, where $m=s|P|$ and $(30,m)=1$.

\end{proof}
%%%%%%%%%%%%%%%%%%%%%%%%%%%%%%%%%%%%%%%%%%%%%%%%%%%%%%
%\section*{Acknowledgments}
%The authors would like to thank the referee for the valuable comments and suggestions which improved the results of the paper.

%%%%%%%%%%%%%%%%%%%%%%%%%%%%%%%%%%%%%
\end{document}